\documentclass[a4paper,oneside,11pt]{article}
\usepackage[english]{babel}
\usepackage[latin1]{inputenc}
\usepackage{indentfirst}
\usepackage{amssymb}
\usepackage{amsmath}
\usepackage{amsthm}
\usepackage{amscd}

\theoremstyle{plain}
 \newtheoremstyle{miestilo}{12pt}{\topsep}{\itshape}{}{\bf}{}{ }{}
\swapnumbers
 \theoremstyle{miestilo}
\newtheorem{theorem}[subsection]{Theorem.}
\newtheorem{proposition}[subsection]{Proposition.}
\newtheorem{lemma}[subsection]{Lemma.}

 \newtheoremstyle{misnotas}{12pt}{12pt}{}{}{\bf}{}{ }{\remark}
 \theoremstyle{misnotas}
  
 \newtheorem{remark}[subsection]{\ {\bf Remark.}}

  \newtheorem{p}[subsection]{\ {\ }}

 \allowdisplaybreaks[2]

\begin{document}

\flushbottom

 \medskip
  {\large\centerline{\bf   Jordan 3-graded Lie algebras with polynomial identities}}

 \medskip

\centerline{ Fernando Montaner\footnote{F. Montaner (fmontane@unizar.es). Partially supported   by grant  PID2021-123461NB-C21, funded
by MCIN/AEI/10.13039/501100011033 and by ``ERDF A way of making Europe''
and
grant E22-23 Álgebra y Geometría, Gobierno de Aragón.).}} \centerline{{\sl Departamento de Matem\'{a}ticas,
Universidad de Zaragoza}} \centerline{{\sl 50009 Zaragoza,
Spain}}   \centerline{and}
\centerline{ Irene Paniello \footnote{I. Paniello (irene.paniello@unavarra.es). Partially supported   by grant  PID2021-123461NB-C21, funded
by MCIN/AEI/10.13039/501100011033 and by ``ERDF A way of making Europe''
and
grant E22-23 Álgebra y Geometría, Gobierno de Aragón.).}} \centerline{{\sl
Departamento de Estad\'{\i}stica, Inform\'{a}tica y  Matem\'{a}ticas, Universidad
P\'{u}blica de Navarra}}
 \centerline{{\sl 31006 Pamplona, Spain}}

 \medskip

\begin{abstract}  We study Jordan 3-graded Lie algebras satisfying 3-graded polynomial identities.
Taking advantage of the Tits-Kantor-Koecher construction, we interpret the PI-condition in terms of their associated Jordan pairs, which allows us to formulate  an analogous of Posner-Rowen Theorem for strongly prime PI Jordan 3-graded Lie algebras. Arbitrary PI Jordan 3-graded Lie algebras are also described by introducing the Kostrikin radical of the Lie algebras. \end{abstract}



\section{Introduction}

One of the main structure theorems of associative PI-algebras is
   Posner theorem for prime associative algebras. The classical, associative, version of this theorem states that any prime ring satisfying a polynomial identity over its centroid is a Goldie ring  and  it has a primitive PI ring of quotients \cite{herstein-carus}.

A Jordan analogous  of Posner, or more accurately of Posner-Rowen, theorem was settled in \cite{pi-ii} for strongly Jordan systems having local algebras  satisfying polynomial identities. This result was named  Posner-Rowen, instead of Posner theorem, since it makes use of the notion of extended central closure  that replaces that of the classical central closure construction for Jordan systems.  This result was later extended in   \cite{hpi},  since as conjectured in \cite{pi-ii}, extended centroids of strongly prime homotope PI Jordan systems coincide with the field of fractions of their centroids.

   In the associative pair setting, the  classical results of the associative PI-theory have been recently extended  in \cite{mp-lma pares} to associative pairs.  These results include, aimed by \cite{pi-i},  the treatment of the notion of PI-element for associative pair,  and also associative pair analogous  of Amitsur (already proved in \cite{id-pares}), Kaplansky and Martindale theorem, and also that of
      Posner-Rowen  theorem for prime associative pairs satisfying homotope polynomial identities,  which is again formulated in terms of central closures of the associative pairs and those of their standard imbeddings.

 In this paper we  address the structure theorems of Jordan 3-graded Lie algebras satisfying 3-graded polynomial identities, more precisely, we consider the
  Jordan 3-graded Lie analogous of Posner-Rowen theorem. Given the available results for strongly prime  homotope PI Jordan pairs proved in \cite{pi-ii}, and the fact that Jordan PI pairs are homotope PI \cite{id-pares}, there are two more issues to be tackled. The first one, the relationship between Jordan 3-graded Lie algebras and Jordan pairs is settled by the Tits-Kantor-Koecher construction \cite{ka1,ka2,ka3,ko1,ko2,tits}, that   provides us with a suitable channel connecting Jordan and 3-graded Lie constructions, as for instance, those of the extended centroids and (extended) central closures. The second issue is to ensure that the considered 3-graded Lie polynomial identities are smoothly transferred to the associated Jordan pairs. To do this we will consider the relationship, again established through the   Tits-Kantor-Koecher construction, between the corresponding Lie and Jordan free objects, to check that there exists a operant version of  3-graded Lie polynomial identity, named essential identity \cite{z-grad}, that allows us to transfer the PI condition from Jordan 3-graded Lie algebras to their associated Jordan pairs.

After this introductory section, and a section of preliminaries, in the third section we settle the relationship between the extended centroid of a nondegenerate Jordan pair and that of its TKK-algebra. The natural isomorphism between both constructions obtained in this section, is considered in section four  to prove the existence of an isomorphism between the central closure of the TKK-algebra  and the TKK-algebra of the extended central closure of a nondegenerate Jordan pair.

We recall here that the notions of extended centroid and central closure were first developed to study prime associative rings satisfying generalized polynomial identities \cite{gpi}, and later extended to the non-associative framework in \cite{baxter-martindale,emo}. The Jordan counterpart of these constructions, named extended centroid and extended central closure  were introduced in \cite{pi-ii}, precisely aimed to study Jordan systems having nonzero local algebras satisfying   polynomial identities. One of the main results in \cite{pi-ii}  is precisely the Jordan version of Posner-Rowen theorem for strongly prime Jordan pairs, formulated in terms of the extended central closure of the Jordan pairs.

In section five we settle the operant version of polynomial identity we consider  here for Jordan 3-graded Lie algebras. Being our purpose to take advantage of the available results for PI Jordan pairs proved in \cite{pi-ii,id-pares}, it becomes necessary to ensure that  the condition of being a PI-algebra can be smoothly transferred   through the TKK-construction, that is, that the 3-graded polynomial identities satisfied by the Jordan 3-graded Lie algebras generate essential identities of their associated Jordan pairs  \cite[0,12]{pi-ii}. To ensure this, we consider the notion of essential polynomial identity \cite{z-grad}, and characterize  essential 3-graded polynomial identities as those 3-graded polynomials of the free 3-graded Lie algebra \cite{neher-libres} not vanishing is some special Lie algebra $sl(n)$. This condition ensures us that such a 3-graded polynomial identity of a Jordan 3-graded Lie algebra produces an essential Jordan polynomial identity for the associated Jordan pair.

The last two sections of this paper are devoted to   study   Jordan 3-graded Lie algebras satisfying essential 3-graded polynomial identities. In section six we assume   these Lie algebras to be strongly prime, and therefore the TKK-algebras of their associated Jordan pairs, which then result   to be also strongly prime \cite{FL-libro}. Then the isomorphism obtained in section four relating TKK-constructions and (extended) central closures  together to  the Jordan version of Posner-Rowen theorem \cite{pi-ii}, provide us with a Jordan 3-graded Lie version of Posner-Rowen theorem for strongly prime Lie algebras. Additionally, considering    all involved objects defined over a base field of characteristic zero or prime al least five,  Zelmanov's description of Lie algebras  with a finite nontrivial $\mathbb{Z}$-grading \cite{z-grad} applies here providing an accurate description of the central closures of the   Jordan 3-graded Lie algebras satisfying essential 3-graded polynomial identities.

Finally  we achieve to describe general Jordan 3-graded Lie algebras satisfying essential 3-graded polynomial identities, under no additional regularity requirements. After discarding the
  strongly primeness of the involved algebras,  we bring into   suitable radical ideals that provide us with nondegenerate quotient Lie algebras and Jordan pairs. This is the case of the Kostrikin radical of a Lie algebra \cite{z-radical}, that we prove here to be related, when the Lie algebra is assumed to be Jordan 3-graded, to the McCrimmon radical of its associated Jordan pair \cite{mc-radical}. McCrimmon radical is the precisely the Jordan  radical that gives rise to    nondegenerate quotient  Jordan pairs.

   Considering the quotient of  a Jordan 3-graded Lie algebra by its Kostrikin radical, that turns easily out to be 3-graded, provides us with a nondegenerate Jordan 3-graded Lie algebra that can be written as a subdirect product of strongly prime Jordan 3-graded Lie algebras, all them
   still satisfying the same essential polynomial identity, and therefore isomorphic to one of the algebras listed in the Jordan 3-graded version of Posner-Rowen theorem obtained in the previous section.

\section{Preliminaries}

\begin{p}
We will work with Jordan pairs and Lie algebras over $\Phi$,    a  unital commutative ring
containing $\frac{1}{2}$ that will be fixed throughout. We
 refer to \cite{jac-lie,loos-jp}
for basic notation, terminology and results on Jordan pairs and Lie algebras.
\end{p}

\begin{p}
A   Jordan pair  $V=(V^+,V^-)$ has products $Q_xy$ for $x\in
V^\sigma$ and $y\in V^{-\sigma}$,\ $\sigma=\pm$, with
linearizations
$Q_{x,z}y=D_{x,y}z=\{x,y,z\}=Q_{x+z}y-Q_xy-Q_zy$.
\end{p}

\begin{p}
   A   Lie algebra   is a $\Phi$-module $L$ with a bilinear  product, denoted   $[x,y]$, satisfying $[x,x]=0$    and the
   Jacobi  identity
  $[[x,y],z]+[[y,z],x]+[[z,x],y]=0$ for all
  $x$,\ $y$,\ $z\in L$.\end{p}

\begin{p}\label{def lie}
  A Lie algebra $L$ is   3-graded  if it admits  a decomposition
  $L=L_1\oplus L_0\oplus L_{-1}$,  where  $L_i$  is a
$\Phi$-submodule of   $L$, for  $i\in\{0,\pm1\}$ and   $[L_i,L_j]\subseteq L_{i+j}$, with  $L_{i+j}=0$ if $i+j\not\in
\{0,\pm1\}$.
A 3-graded Lie algebra   $L=L_1\oplus L_0\oplus L_{-1}$ is    Jordan 3-graded  if $[L_1,L_{-1}]=L_0$, and the
pair  $(L_1,L_{-1})$ admits a
 Jordan pair structure defined  by
$\{x,y,z\}=[[x,y],z]$, for any $x$,\, $z\in L_\sigma$,\, $y\in L_{-\sigma}$,\, $\sigma=\pm$.
Then     $V=(L_1,L_{-1})$ is  called the     Jordan
pair associated to $L$ \cite[1.5]{neher-118}.
\end{p}

\begin{p} Since  ${1\over2}\in \Phi$, by  \cite[Proposition 2.2(a)]{loos-jp} and  \ref{def lie}(ii) above,
  any Jordan
 3-graded Lie  algebra $L$   determines  a  Jordan product on $V=(L_1,L_{-1})$  given by $Q_xy=\frac{1}{2}
\{x,y,x\}=\frac{1}{2}  [[x,y],x]$. If, moreover,
   ${1\over6}\in \Phi$,    by   \cite[Proposition~2.2(b)]{loos-jp},   any 3-graded Lie algebra  $L=L_1\oplus L_0\oplus L_{-1}$ defines   a Jordan pair structure on   $(L_1,L_{-1})$. Indeed,  if  ${1\over6}\in \Phi$,  any pair of $\Phi$-modules   $(L_1,L_{-1})$ endowed  with trilinear mappings
\[\begin{array}{ccl}\{\ ,\ ,\ \}: L_\sigma  \times L_{-\sigma}\times L_\sigma  &\to  & L_\sigma\\
\qquad\qquad (x,y,z) & \mapsto &\{x,y,z\}=D(x,y)z\end{array}\]
satisfying
\begin{enumerate}
\item[(i)] $\{x,y,z\}=\{z,y,x\}$,%
\item[(ii)] $[D(x,y),D(u,v)]=D(\{x,y,u\},v)-D(u,\{y,x,v\})$, \end{enumerate}
for all  $x$,\ $z$,\ $u\in L_\sigma$, $y$,\ $v\in L_{-\sigma}$, $\sigma=\pm$, has a Jordan pair structure.
\end{p}

\begin{p} Any Jordan pair $V$ defines   a
 Jordan 3-graded Lie algebra  given by the Tits-Kantor-Koecher  construction, that for
any Jordan pair  $V=(V^+,V^-)$ gives rise to a 3-graded Lie algebra  $L=L_1\oplus L_0\oplus L_{-1}$  with $L_1=V^+$ and
$L_{-1}=V^-$.  These Lie algebras were (independently) introduced by
   Kantor \cite{ka1,ka2,ka3},
  Koecher \cite{ko1,ko2} and Tits \cite{tits}, and  also studied by Meyberg in \cite{meyberg}. We refer to \cite[11.2]{FL-libro} for accurate details on this construction and write $TKK(V)=V^+\oplus \hbox{IDer} V\oplus V^-$ for the  TKK-Lie algebra of a Jordan pair $V$, where
  $\hbox{IDer} V$ is the ideal of the Lie algebra of derivations $\hbox{Der} V$ of $V$ generated by the inner derivations
    $\delta(x,y):=(D_{x,y},-D_{y,x})$   of $V$,
 for all
 $x\in V^+$, $y\in V^-$.
 \end{p}

 \begin{p}\label{central extension} The relationship between
 Jordan 3-graded Lie algebras and   TKK-algebras was settled by Neher in
  \cite[1.5(6)]{neher-118}. For   any
  Jordan 3-graded Lie algebra $L$ with associated Jordan pair $V$, it holds that
  $TKK(V)\cong L/C_V$,  where
$C_V=\{x\in L_0\mid [x,L_1]=0=[x,L_{-1}]\}=Z(L)\cap L_0$ and $Z(L)=\{x\in L\mid [x,L]=0\}$ denotes the center of $L$.
 \end{p}

\begin{p} A Lie algebra $L$ is    $\mathbb{Z}$-graded  if there exists a decomposition
   $L= \sum_{n\in\mathbb{ Z}} L_n$,   where   $L_n$  is a $\Phi$-submodule of $L$  for all
   $n\in\ \mathbb{Z}$,  and $[L_i,L_j]\subseteq L_{i+j}$  for  $i$,\ $j\in \mathbb{Z}$. Such a grading is
   finite if the set $\{n\in \mathbb{Z}\mid
 L_n\neq0\}$ is finite, and   nontrivial if $\sum_{n\neq0}L_n\neq0$. TKK-algebras of Jordan pairs are
  $\mathbb{Z}$-graded Lie algebras
   $L=L_{-1}+L_0+L_1$, with $L_{-1}=V^-$,\ $L_0=\hbox{IDerV}$,\  $L_1=V^+$  and
         $L_i=0$ for $\vert i\vert>1$ \cite{z-grad}.
\end{p}

\begin{p}
An ideal $I$ of a Jordan pair $V$ is called  essential    if it has nontrivial intersection with any nonzero ideal of $V$.
Essential ideals of nondegenerate Jordan pairs are those having zero annihilator.
Similarly,  an ideal $I$ of a Lie algebra $L$ is    essential  if and only if it has nontrivial intersection with any nonzero
ideal of $L$. Essential ideals of semiprime Lie algebras are those with zero centralizer.  Regularity conditions, such as (semi)primeness, nondegenerancy  or simplicity can  be
transferred between Jordan pairs and their TKK-algebras \cite[Proposition 11.25]{FL-libro}.
\end{p}

\begin{p}
The notions of extended centroid and central closure were first introduced to study prime rings satisfying generalized polynomials
identities \cite{gpi}, and later   generalized to nonassociative rings   \cite{baxter-martindale,emo}.

Let $L$ be a semiprime Lie algebra. We denote by $Ad(L)$ the subalgebra of $End_\Phi(L)$ generated by $ad(L)=\{ad\, x\mid x\in L\}$. A pair $(f,I)$ is a
  permissible map of $L$  if $I$ is an essential ideal of $L$ and $f:I\to L$ is an additive map that commutes with all elements of
$Ad(L)$.  Following  \cite{baxter-martindale,emo}  we will say that two  permissible maps $(f_1,I_1)$ and $(f_2,I_2)$ of a Lie algebra $L$ are equivalent  if there exists an essential ideal $I$ of
$L$, contained into $I_1\cap I_2$, such that $f_1(x)=f_2(x)$ for all $x\in I$.  This defines an equivalence relation on the set of permissible maps of $L$. We
will denote by $\overline{(f,I)}$ the equivalence class determined by the permissible map $(f,I)$, and by  ${\cal C}(L)$ the set of all
equivalence classes of permissible maps of $L$. The set  ${\cal C}(L)$   is called the  extended centroid of $L$. If $L$ is a
  a semiprime Lie algebra, then ${\cal C}(L)$ is a von Neumann regular unital algebra \cite[Theorem 2.5]{baxter-martindale}.
\end{p}

\begin{p}  Extended centroids of quadratic Jordan systems were introduced in
   \cite{pi-ii}.
Let $V$ be a Jordan pair and   $U=(U^+,U^-)$   an ideal of $V$. A pair $g=(g^+,g^-)$ of linear mappings  $g^\sigma:U^\sigma\to
V^\sigma$, $\sigma=\pm$, is a {\it $V$-homomorphism} of  $V$ if for all  $y^\sigma\in U^\sigma$,\ $x^\sigma$,\
$z^\sigma\in V^\sigma$:
\begin{enumerate}
 \item[(a)]\ $g^\sigma(Q_{x^\sigma}y^{-\sigma})=Q_{x^\sigma}g^{-\sigma}(y^{-\sigma})$,
\item[(b)]\ $g^\sigma(Q_{I^\sigma}V^{-\sigma})\subseteq I^\sigma$ and
$(g^\sigma)^2(Q_{y^\sigma}x^{-\sigma})=Q_{g^\sigma(y^\sigma)}x^{-\sigma}$,
\item[(c)]\ $g^\sigma(\{y^\sigma,z^{-\sigma},x^\sigma\})=\{g^\sigma(y^\sigma),z^{-\sigma},x^\sigma\}$.\end{enumerate}
 We denote by  $Hom_V(U,V)$ the set of all $V$-homomorphisms from $U$ into $V$ \cite[1.1]{pi-ii}.
 A pair $(g,U)$ is a   permissible map of $V$  if $U$ is an essential ideal of $V$.

Two permissible maps  $(g_1,U_1)$ and $(g_2,U_2)$ of a Jordan pair $V$ are equivalent,  denoted by $(g_1,U_1)\sim(g_2,U_2)$,
  if there exists an essential
 ideal $U$ of $V$, contained into $U_1\cap U_2$, such that $g_1^\sigma(x)=g_2^\sigma(x)$ for all $x\in U^\sigma$,\ $\sigma=\pm$.
 Again, this defines
  an equivalence relation, with classes denoted by $[g,U]$,  on the set of all permissible maps of $V$, with
 quotient set  ${\cal C}(V)$    called the   extended centroid of $V$.
The extended centroid ${\cal C}(V)$ of a nondegenerate Jordan pair $V$  is a commutative, associative, unital (von Neumann) regular
algebra \cite[Theorem 1.15,  Proposition 2.7]{pi-ii}.
\end{p}

\begin{p}
Let $(g,U)$ be a permissible map of a   Jordan pair   $V$, and   $K$   a nonzero ideal of $V$ contained in $U$. We will say that
  $g$ restricts to  $K$ if $g_K\in Hom_V(K,V)$,  where  $g_K$ denotes the
restriction of  $g$ to $K$. A necessary and sufficient condition for $g$ to restrict to $K$ is   $g^\sigma(Q_{K^\sigma}
V^{-\sigma})\subseteq K^\sigma$, $\sigma=\pm$ \cite[p.~484]{pi-ii}.
\end{p}

\begin{p} The corresponding escalar extension given by the extended centroid is called central closure. We refer the reader to \cite{baxter-martindale} for the construction of the central closure  ${\cal C}(L)L$ of a semiprime Lie algebra $L$, and to \cite{pi-ii} for that of the  extended central closure  ${\cal C}(V)V$ of a  nondegenerate Jordan pair $V$.
\end{p}

 \section{The extended centroid of the TKK-algebra of a nondegenerate Jordan pair}

This section is aimed to study   the relationship   between the  extended centroid
   of a nondegenerate Jordan pair $V$  and that of its
TKK-algebra $TKK(V)$.

\begin{p}\label{ideales-relacion-p} Let $L$ be a  Jordan 3-graded Lie algebra  with associated
Jordan pair~$V$. The following assertions are straightforward:\begin{enumerate} \item[(i)] For any ideal $I$ of
$L$,  $I\cap V=(I\cap V^+,I\cap V^-)$  and
$\pi(I)=(\pi_+(I),\pi_-(I))$ are ideals of  $V$  such that
$\pi(I)^3\subseteq I\cap V\subseteq \pi(I)$, where $\pi_\sigma$
denotes the canonical projections of $L$ on $L_\sigma$,\
$\sigma=\pm$.
  Thus, if  $V$ is semiprime and $I$ nonzero,  then   $I\cap V$ is a nonzero
ideal of $V$.
\item[(ii)] If  $U=(U^+,U^-)$ is a nonzero ideal of $V$, then
  ${\cal I}(U)=U^+\oplus([U^+,V^-]+[V^+,U^-])\oplus U^-$ is a nonzero ideal of $L$.\end{enumerate}\end{p}

 \begin{lemma}\label{ideales-relacion-lema} Let   $V$  be a semiprime Jordan
 pair.
\begin{enumerate}\item[(i)]
 If $I$ is an essential ideal  of   $TKK(V)$,   then  $I\cap V$ is an
essential ideal of $V$.
\item[(ii)]  If    $U=(U^+,U^-)$ is an essential ideal of
  $V$, then ${\cal I}(U)$ is an essential ideal of $ TKK(V)$.\end{enumerate}
\end{lemma}

\begin{proof}  (i) Let $K$ be a nonzero ideal of $V$. Then, since  $  {\cal I}(K)$ is a nonzero ideal of $TTK(V)$, by the essentiality of $I$,
 we have that $I\cap {\cal I}(K)$
is a nonzero ideal of  $TKK(V)$. Therefore,
   by \ref{ideales-relacion-p}(i), it holds that  $0\neq I\cap {\cal I}(K)\cap V=(I\cap V)\cap
  K$.  Hence $I\cap V$ is an  essential ideal of $V$.

(ii)  Let  now $K$ be a nonzero ideal of $L$. By \ref{ideales-relacion-p}(i),
  $K\cap V$ is a nonzero ideal of   $V$ and thus, by the essentiality of $U$,    it holds that  $U\cap(K\cap V)=U\cap K\neq0$.
Therefore
 $0\neq U^\sigma\cap K\subseteq {\cal I}(U)\cap K$, for some $\sigma=\pm$, implying that
 ${\cal I}(U)$  is an essential ideal of $ TKK(V)$.  \end{proof}

\begin{remark}\label{esencial-nodeg}
\begin{enumerate}
\item[(i)] Lemma \ref{ideales-relacion-lema} applies, in particular,    to
  nondegenerate Jordan pairs, since,   by  \cite[p. 212]{da-hermitian},   nondegenerate Jordan pairs are   semiprime.
\item[(ii)] Taking advantage of the 3-grading of  the TKK-algebras, it follows from
Lemma~\ref{ideales-relacion-lema}(i), that any essential ideal $I$ of the TKK-algebra of a semiprime Jordan pair $V$, contains an essential ideal of the form
 ${\cal I}(U)$, for an essential ideal $U$ of $V$. Indeed, it suffices to consider $U= I\cap V$.
\end{enumerate}
    \end{remark}

\begin{p}\label{def producto} Given two ideals $K=(K^+,K^-)$ and
$I=(I^+,I^-)$ of a Jordan pair~$V$,   the product
 $K\ast
I=(Q_{K^+}I^-+Q_{V^+}Q_{K^-} I^+,Q_{K^-}I^++Q_{V^-}Q_{K^+}I^-)$ is
an ideal of $V$. If $K=V$,  then we have $V\ast
I=(Q_{V^+}I^- ,Q_{V^-}I^+ )$ \cite[p.~221]{mc-inh}.
\end{p}

 \begin{lemma}\label{cambio-ideal}  Let $I$ be an ideal of the $TKK$-algebra  of a Jordan pair $V$.
 Then $\widetilde{I}={\cal I}\big(V\ast (\pi_+(I),\pi_-(I))\big)$, where   $\pi_\sigma$
  denotes  the canonical
  projections  from $TKK(V)$ onto  $TKK(V)_\sigma$,   $\sigma\in\{0,\pm \}$,  is an
ideal of  $TKK(V) $ such that:
\begin{enumerate}
\item[(i)]  $\widetilde{I}\subseteq I $,
\item[(ii)]
  $\widetilde{I}\cap V=V\ast(\pi_+(I),
\pi_-(I))=\pi(\widetilde{I})$.
\end{enumerate} Moreover, if  $V$ a is
nondegenerate Jordan pair, then
  $I$ is essential if and only if   $\widetilde{I}$ is essential. \end{lemma}

 \begin{proof} Let $I$ be an ideal of   $TKK(V)$. We first note that, by \ref{ideales-relacion-p}(i),  $\pi(I)=(\pi_+(I),\pi_-(I))$  is an ideal of $V$  that contains
 $I\cap V$.

 Take now   $l\in TKK(V)$  and elements
  $x_\sigma$,\ $z_\sigma\in TKK(V)_\sigma$,\,   $\sigma=\pm$.  Since  by the 3-grading of $TKK(V)$ we have
$ \{x_\sigma,\pi_{-\sigma}(l),z_\sigma
\}=[[x_\sigma,\pi_{-\sigma}(l)],z_\sigma ]= [[x_\sigma,l],z_\sigma]
$, it follows from  \ref{def producto} that
   $ V\ast(\pi_+(I),\pi_-(I))$ is an ideal of $V$ such that:
\begin{align*}
 V\ast(\pi_+(I),\pi_-(I))&=\big(Q_{V^+}\pi_-(I),Q_{V^-}\pi_+(I)\big)=\\
 &=\big(\{V^+,\pi_-(I),V^+\},\{V^-,\pi_+(I),V^-\}\big)=\\
& = \big([V^+,[V^+,I]],[V^-,[V^-,I]]\big).\end{align*}
By  \ref{ideales-relacion-p}(ii),   $\widetilde{I} = {\cal
I}\big(V\ast(\pi_+(I),\pi_-(I))\big)$
   is an  ideal of $TKK(V)$,   clearly contained in $I$, and  by the 3-grading of   $\widetilde{I}$ it holds that
   $\widetilde{I}\cap
V=V\ast(\pi_+(I), \pi_-(I))$. Hence  $ \widetilde{I}\cap V=V\ast(\pi_+(I),
\pi_-(I))=\pi(\widetilde{I}) $.

Assume now that the Jordan pair $V$ is   nondegenerate, and let  $I$  be an essential ideal of $TKK(V)$. Then, by Lemma \ref{ideales-relacion-lema},
 $I\cap V$  is an essential ideal of $V$  and,
 since
 $I\cap V\subseteq \pi(I)$,  it follows that   $\pi(I)$ is also
essential in $V$.  Then, by  \cite[Lemma 1.2(a)]{pi-ii},  $V\ast
(\pi_+(I),\pi_-(I))$ is  an  essential ideal of $V$, and
by Lemma~\ref{ideales-relacion-lema}(ii) we obtain that
 $\widetilde{I}$ is essential  in $TKK(V)$.
 The converse   follows from the fact that
$\widetilde{I}\subseteq I$.  \end{proof}

 \begin{remark}\label{cons-cambio-ideal} Let $V$ be a nondegenerate Jordan pair, and let    $\lambda=\overline{(f,I)}\in {\cal
 C}(TKK(V))$. By Lemma \ref{cambio-ideal}, the pair
 $(f_{ \widetilde{I}}, \widetilde{I})$, where   $\widetilde{I}={\cal I}\big(V\ast (\pi_+(I),\pi_-(I))\big)$, is a permissible map such that
 $\lambda=\overline{(f,I)}=\overline{(f_{ \widetilde{I}}, \widetilde{I})}$  \cite[Corollary 2.3]{baxter-martindale}.
Thus, since $f\big([V^\sigma,[V^\sigma,I]]
\big)=[V^\sigma,[V^\sigma,f(I)]]\subseteq V^\sigma$ and
\begin{align*}&f\big([[V^+,[V^+,I]],V^-]+ [V^+,
[V^-,[V^-,I]]]\big)=\\
&=[[V^+,[V^+,f(I)]],V^-]+ [V^+, [V^-,[V^-,f(I)]]]\subseteq
[V^+,V^-], \end{align*}
   replacing $(f,I)$ by $(f_{ \widetilde{I}},
\widetilde{I})$, if necessary,  for any
$\lambda=\overline{(f,I)}\in {\cal
 C}(TKK(V))$  we will
assume that $I$ is a 3-graded essential  ideal of $TKK(V)$ such that
$f(I^\sigma)\subseteq TKK(V)_\sigma$, \, $\sigma\in\{0,\pm1\}$, that is, $f$ is a 3-graded map.
\end{remark}

 \begin{proposition}\label{aplicacion-centroides} Let $ TKK(V)$ be the TKK-algebra of a  nondegenerate Jordan pair $V$. Then
  the map:
 \[\begin{array}{ccl}\Psi: {\cal C}(TKK(V)) & \to & {\cal C}(V)\\
\qquad \overline{(f,I )} &\mapsto &[f_{  I\cap V}, I\cap
V]\end{array}\]
 defines a ring homomorphism from the extended centroid  ${\cal C}(TKK(V))$ of $TKK(V)$ to the extended centroid ${\cal C}(V)$  of  $V$.
   \end{proposition}

 \begin{proof} Let $\lambda=\overline{(f,I)}\in{\cal C}(TKK(V))$, where $(f,I)$ is a permissible map of $TKK(V)$ as in
  Remark~\ref{cons-cambio-ideal}, so that $I$ is a 3-graded  essential ideal of $TKK(V)$ and $f(I^\sigma)\subseteq V^\sigma$ for $\sigma=\pm$.

Let us denote $f^\sigma=f_{I^\sigma}$, and take $x^\sigma, z^\sigma\in V^\sigma$ and $y^\sigma\in I^\sigma$, for $\sigma=\pm$. Clearly $f_{I\cap V}$ is a linear map, and since we are assuming  $I\cap V^\sigma=\pi_\sigma(I)$, it holds that:
  \begin{align*}  f^\sigma(Q_{x^\sigma}y^{-\sigma}) &=f^\sigma\Big({1\over2}[[x^\sigma,y^{-\sigma}],x^\sigma]\Big)=
    f\Big({1\over2}[[x^\sigma,y^{-\sigma}],x^\sigma]\Big)=\\
&={1\over2}[[x^\sigma,f(y^{-\sigma})],x^\sigma]=
     {1\over2}[[x^\sigma,f^{-\sigma}(y^{-\sigma})],x^\sigma]= Q_{x^\sigma}f^{-\sigma}(y^{-\sigma}),\end{align*}
Moreover, since:
\begin{align*} f^\sigma\big(Q_{I^\sigma} V^{-\sigma}\big)&\subseteq f^\sigma\big(\big[[I^\sigma,V^{-\sigma}],I^\sigma\big]\big)=
   f\big(\big[\big[I^\sigma,V^{-\sigma}\big],I^\sigma\big]\big)=\\
&= \big[\big[f(I^\sigma),V^{-\sigma}\big],I^\sigma\big]
  \subseteq \big[\big[  V^\sigma ,V^{-\sigma}\big],I^\sigma\big]=  \big\{V^\sigma ,V^{-\sigma},I^\sigma\big\}\subseteq I^\sigma, \end{align*} it follows that:
  \begin{align*}(f^\sigma)^2 (Q_{y^\sigma}x^{-\sigma} )&=(f^\sigma)^2({1\over2}[[y^\sigma,x^{-\sigma}],y^\sigma])=
f^2({1\over2}[[y^\sigma,x^{-\sigma}],y^\sigma])=\\
&= {1\over2}f ([[f (y^\sigma),x^{-\sigma}],y^\sigma]) ={1\over2}[[f(y^\sigma),x^{-\sigma}],f(y^\sigma)]= \\
&={1\over2}[[f^\sigma(y^\sigma),x^{-\sigma}],f^\sigma(y^\sigma)]
 =Q_{f^\sigma(y^\sigma)}x^{-\sigma},\end{align*}
and, finally,  we also have:
\begin{align*}f^\sigma\big(\{y^\sigma,z^{-\sigma},x^\sigma\}\big)&=f^\sigma\big([[y^\sigma,z^{-\sigma}],x^\sigma]\big)=
f\big([[y^\sigma,z^{-\sigma}],x^\sigma]\big)=\\
&=[[f(y^\sigma),z^{-\sigma}],x^\sigma]= [[f^\sigma(y^\sigma),z^{-\sigma}],x^\sigma] =\{f^\sigma(y^\sigma),z^{-\sigma},x^\sigma\}.\end{align*}
Hence $(f_{ I\cap V},I\cap V)$ is a $V$-homomorphism, and therefore, by the essentiality of $I\cap V$, it is a  permissible map of
   $V$.

We next claim that $\Psi$ is
well-defined.  To prove this claim,  let
  $(f_1,I_1)$ and $(f_2,I_2)$ be permissible maps of $TKK(V)$  such that
$\overline{(f_1,I_1)}= \overline{(f_2,I_2)}$ in ${\cal C}(TKK(V))$.
By Remark~\ref{cons-cambio-ideal} and \cite[Corollary 2.3]{baxter-martindale}, we can assume that
$I=I_1\cap I_2$ is an  essential 3-graded ideal of $TKK(V)$ such that
$(f_1)_I=(f_2)_I$. Therefore,
 for all  $x^\sigma\in I\cap V^\sigma$  it holds that   $
(f_1)^\sigma(x^\sigma)=f_1(x^\sigma)=f_2(x^\sigma)=(f_2)^\sigma(x^\sigma)$,  that results into
       $\big((f_1)_{ I_1\cap V}, I_1\cap V\big)$ and
$\big((f_2)_{ I_2\cap V}, I_2 \cap V\big)$ being equivalent
permissible maps of $V$ by   \cite[1.4]{pi-ii}.  Thus    $\Psi$ is a
well-defined map.

Let now
 $\lambda_1=\overline{(f_1,I_1)}$ and
$\lambda_2=\overline{(f_2,I_2)}\in{\cal C}(TKK(V))$. By
\cite[p.~1108]{baxter-martindale} and   Remark \ref{cons-cambio-ideal},
 $\lambda_1+\lambda_2=\overline{\big((f_1)_{  I_1\cap I_2}+(f_2)_{  I_1\cap I_2},I_1\cap
 I_2\big)}$,    and therefore,
 $$\Psi(\lambda_1+\lambda_2)=\big[\big((f_1)_{  I_1\cap I_2}+(f_2)_{  I_1\cap I_2}\big)_{  I_1\cap I_2\cap V},I_1\cap I_2\cap
 V\big],  $$ whereas, on the other hand, by \cite[1.11]{pi-ii},
\begin{align*}\Psi(\lambda_1)+\Psi(\lambda_2)&=\big[(f_1)_{  I_1\cap V},I_1\cap V\big]+ \big[(f_2)_{  I_2\cap V},I_2\cap
V\big]=\\
&=\big[(f_1)_K+(f_2)_K,K\big], \end{align*}
where, by \cite[Lemma 1.2(a)]{pi-ii},  $K=(I_1\cap I_2\cap V)\ast V$ is an essential ideal of  $V$
contained in  $ I_1\cap I_2\cap V$. Therefore,
  since $f_i(I_i^\sigma)\subseteq V^\sigma$,\,  $\sigma=\pm$,\
$i=1,2$, for all  $k^\sigma\in K^\sigma$ we have
 $\big((f_1)_{  I_1\cap I_2}+(f_2)_{  I_1\cap I_2}\big)_{I_1\cap I_2\cap
 V}(k^\sigma)=\big((f_1)_K+(f_2)_K\big)(k^\sigma)$ and,
 consequently,
   $\Psi(\lambda_1+\lambda_2)=\Psi(\lambda_1)+\Psi(\lambda_2)$  by \cite[Lemma 1.10]{pi-ii}. Hence the map $\Psi$ is additive.

 To prove that $\Psi$  is a multiplicative map, we note that by
\cite[p.~1108]{baxter-martindale}   $ \lambda_1 \lambda_2=
\overline{\big(f_1f_2,f_2^{-1}(I_1)\big)}  $  and therefore
 $\Psi(\lambda_1\lambda_2)=\big[(f_1f_2)_{f_2^{-1}(I_1)\cap V}, f_2^{-1}(I_1)\cap V\big]$, whereas,
  by \cite[1.13]{pi-ii},  we have:
  $$\Psi(\lambda_1)\Psi(\lambda_2) =\big[(f_1)_{  I_1\cap V}, I_1\cap V\big]\big[(f_2)_{  I_2\cap V}, I_2\cap V\big]=  \big[(f_1f_2)_{K\ast V}, K\ast V\big],$$ for the essential ideal
$K=(I_1\cap I_2\cap V)\ast V$  of $V$
\cite[Lemma 1.2(a)]{pi-ii}.
 We
  claim that   both $\Psi(\lambda_1\lambda_2)$ and
$\Psi(\lambda_1)\Psi(\lambda_2)$ can be restricted to the essential ideal
$(f_2^{-1}(I_1)\cap V)\cap (K\ast V)$ of $V$.

To prove this claim we
first note that, by  Remark \ref{cons-cambio-ideal},  $f_2(I_2^\sigma)\subseteq V^\sigma$ which implies that
   $f_2^{-1}(I_1)\cap V^\sigma\subseteq I_2^\sigma$, and since
  $(I_1\cap I_2\cap V)^\sigma=I_1^\sigma\cap
I_2^\sigma$, we also have
 $(K\ast  V)^\sigma \subseteq I_1^\sigma\cap I_2^\sigma$. Then:
 $$\Big((f_2^{-1}(I_1)\cap V)\cap (K\ast V)\Big)^\sigma =(f_2^{-1}(I_1)\cap V^\sigma)\cap (K\ast V)^\sigma=  f_2^{-1}(I_1) \cap (K\ast V)^\sigma,$$
and it follows that:
\begin{align*}&(f_1f_2)^\sigma\big(Q_{f_2^{-1}(I_1) \cap (K\ast V)^\sigma} V^{-\sigma}\big)=\\
 &=(f_1f_2)^\sigma\big(\big[\big[f_2^{-1}(I_1) \cap (K\ast V)^\sigma,V^{-\sigma}\big],%
 f_2^{-1}(I_1) \cap (K\ast V)^\sigma\big]\big)=\\
 &=(f_1f_2) \big(\big[\big[f_2^{-1}(I_1) \cap (K\ast V)^\sigma,V^{-\sigma}\big],%
 f_2^{-1}(I_1) \cap (K\ast V)^\sigma\big]\big)=\\%
 &=
  f_1   \big(\big[\big[f_2(f_2^{-1}(I_1) \cap (K\ast V)^\sigma),V^{-\sigma}\big],%
  f_2^{-1}(I_1) \cap (K\ast V)^\sigma\big]\big)\subseteq \\%
&\subseteq
  f_1   \big(\big[\big[f_2(f_2^{-1}(I_1)\cap V^\sigma) ,V^{-\sigma}\big],%
  f_2^{-1}(I_1) \cap (K\ast V)^\sigma\big]\big)\subseteq \\%
&\subseteq   f_1\big(\big[\big[I_1\cap V^\sigma,V^{-\sigma}\big],%
  f_2^{-1}(I_1) \cap (K\ast V)^\sigma\big]\big)=\\
&=  f_1\big(\big[\big[I_1^\sigma,V^{-\sigma}\big],%
 f_2^{-1}(I_1) \cap (K\ast V)^\sigma]\big)=\\
&=   \big[\big[f_1(I_1^\sigma),V^{-\sigma}\big], f_2^{-1}(I_1) \cap (K\ast V)^\sigma\big] \subseteq \\ %
  &\subseteq \big[\big[  V^\sigma,V^{-\sigma}\big], f_2^{-1}(I_1) \cap (K\ast V)^\sigma\big]=\\
 & =\{ V^\sigma,V^{-\sigma} , f_2^{-1}(I_1) \cap (K\ast V)^\sigma\}\subseteq
  f_2^{-1}(I_1) \cap (K\ast V)^\sigma,\end{align*}
implying that $\Psi(\lambda_1\lambda_2)$ can be restricted to the ideal
$(f_2^{-1}(I_1)\cap V)\cap (K\ast V)$ of $V$.
 On the other hand, for  $\Psi(\lambda_1)\Psi(\lambda_2)$  we have:
\begin{align*}&((f_1)^\sigma(f_2)^\sigma)\big(Q_{ f_2^{-1}(I_1) \cap (K\ast V)^\sigma}V^{-\sigma}\big)=\\%
&=((f_1)^\sigma(f_2)^\sigma)\big(\big[\big[f_2^{-1}(I_1) \cap (K\ast V)^\sigma,V^{-\sigma}\big],f_2^{-1}(I_1) \cap (K\ast
V)^\sigma\big]\big)=\\
 &= (f_1)^\sigma\big((f_2)^\sigma \big(\big[\big[f_2^{-1}(I_1) \cap (K\ast V)^\sigma,V^{-\sigma}\big],
f_2^{-1}(I_1) \cap (K\ast V)^\sigma\big]\big)\big)=\\%
 &= (f_1)^\sigma\big( f_2  \big(\big[\big[f_2^{-1}(I_1) \cap (K\ast
V)^\sigma,V^{-\sigma}\big], f_2^{-1}(I_1) \cap (K\ast V)^\sigma\big]\big)\big)\subseteq\\%
 &\subseteq (f_1)^\sigma\big(
\big[\big[f_2(f_2^{-1}(I_1) \cap  V^\sigma),V^{-\sigma}\big], f_2^{-1}(I_1) \cap (K\ast V)^\sigma\big]\big) \subseteq\\ &\subseteq
(f_1)^\sigma\big(  \big[\big[I_1\cap V^\sigma,V^{-\sigma}\big], f_2^{-1}(I_1) \cap (K\ast V)^\sigma\big]\big) =\\ &=  f_1 \big(
\big[\big[I_1\cap V^\sigma,V^{-\sigma}\big], f_2^{-1}(I_1) \cap (K\ast V)^\sigma\big]\big) =\\ &=
\big[\big[f_1(I_1^\sigma),V^{-\sigma}\big], f_2^{-1}(I_1) \cap (K\ast V)^\sigma\big]  \subseteq\\%
 &\subseteq \big[\big[  V ^\sigma
,V^{-\sigma}\big],f_2^{-1}(I_1) \cap (K\ast V)^\sigma\big] =\\
&=\big\{V^\sigma,V^{-\sigma},f_2^{-1}(I_1) \cap (K\ast V)^\sigma\big\} \subseteq f_2^{-1}(I_1) \cap (K\ast V)^\sigma. \end{align*}

 Therefore   both   $\Psi(\lambda_1\lambda_2)$ and  $  \Psi(\lambda_1)\Psi(\lambda_2)$  are permissible maps of $V$ that  can be restricted
 to the essential ideal
 $(f_2^{-1}(I_1)\cap V)\cap (K\ast V)$. Moreover,
  since
  $f_i(I_i^\sigma)=f_i^\sigma(I^\sigma)$,   both
   $\Psi(\lambda_1\lambda_2)$ and $ \Psi(\lambda_1)\Psi(\lambda_2)$ agree in  $(f_2^{-1}(I_1)\cap V)\cap (K\ast V)$. This implies that
  $\Psi$  is a multiplicative map and, consequently,  a ring homomorphism.
\end{proof}

 \begin{theorem}\label{isomorfismo-centroides}  The extended centroid ${\cal C}(V)$ of a nondegenerate Jordan pair $V$
 is isomorphic to the extended centroid   ${\cal C}(TKK(V))$ of its
TKK-algebra $TKK(V)$. \end{theorem}

 \begin{proof} Let $V$ be a  nondegenerate Jordan pair and let us define the map
\[\begin{array}{ccc}\Upsilon:\ {\cal C}(V) & \to & {\cal C}(TKK(V))\\
\,[(g^+,g^-),(U^+,U^-)] &\mapsto & \overline{(f, {\cal
I}(U))}\end{array}\]
where
 $f: {\cal I}(U)\to TKK(V)$ is defined  by
  \begin{align*} &f\Big(u^++\big(\sum_{i=1}^n[u_i^+,v_i^-]
  +\sum_{j=1}^m[v_j^+,u_j^-]\big)+u^-\Big)=\\
  &=g^+(u^+)+
  \Big(\sum_{i=1}^n[g^+(u_i^+),v_i^-]+\sum_{j=1}^m[v_j^+,g^-(u_j^-)]\Big)+ g^-(u^-)\end{align*} for all $u^\sigma$,\ $u^\sigma_i\in U^\sigma$,
  \
   $v^\sigma_i\in
  V^\sigma$, $\sigma=\pm$.

To prove that  $f: {\cal I}(U)\to TKK(V)$ is well-defined, let  $u_i^\sigma$,\ $u_j^\sigma\in
U^\sigma$,\  $v_i^\sigma$,\,  $v_j^\sigma\in V^\sigma$ such that $
\sum_{i=1}^n[u_i^\sigma,v_i^{-\sigma}]
+\sum_{j=1}^m[v_j^\sigma,u_j^{-\sigma}]=0$, and write   $
a=\sum_{i=1}^n[g^\sigma
(u_i^\sigma),v_i^{-\sigma}]+\sum_{j=1}^m[v_j^\sigma,g^{-\sigma
}(u_j^{-\sigma})]$. Then, for all       $w^\sigma\in V^\sigma$,  $\sigma=\pm$:
\begin{align*}
  [a,w^\sigma]&=\sum_{i=1}^n\{g^\sigma (u_i^\sigma),v_i^{-\sigma},w^\sigma\}+\sum_{j=1}^m\{v_j^\sigma,g^{-\sigma}(u_j^{-\sigma}),w^\sigma\}=\\
 &=g^\sigma\big(\sum_{i=1}^n\{ u_i^\sigma,v_i^{-\sigma},w^\sigma \}+\sum_{j=1}^m\{v_j^\sigma ,u_j^{-\sigma},w^\sigma \}\big)=\\
&=g^\sigma \big(\big[ \sum_{i=1}^n[u_i^\sigma ,v_i^{-\sigma}]
+\sum_{j=1}^m[v_j^\sigma ,u_j^{-\sigma}],w^\sigma
\big]\big)=g^\sigma ([0,w^\sigma])=g^\sigma (0)=0.
\end{align*}
which implies   $a\in Z(TKK(V))$. But since,  by  \cite[Proposition 11.25]{FL-libro},    TKK-algebras of nondegenerate Jordan pairs are centerless, it follows that
  $a=0$. Hence the map  $f: {\cal I}(U)\to TKK(V)$ is well-defined.

Now, since, by Lemma \ref{cambio-ideal}(ii), ${\cal
I}(U) $  is an essential ideal of $V$, to prove that
  $ (f, {\cal
I}(U)) $ is a permissible map of $TKK(V)$, it suffices to check that $[f,ad\, y] (  {\cal
I}(U))=0$ for all $y\in TKK(V)$. Write $y=y^++y_0+y^-$, with
$ y_0 =\sum_{i=1}^n\delta(a_i^\sigma,b_i^{- \sigma})=\sum_{i=1}^n\big(D(a_i^\sigma,b_i^{- \sigma}),-D(b_i^{- \sigma},a_i^\sigma)\big) $
for some  $a_i^\sigma\in V^\sigma$ and  $b_i^{- \sigma}\in V^{- \sigma}$. Then, for all $u^\sigma\in U^\sigma$, $\sigma=\pm$,
 $ [f,ad\,
y^\sigma](u^\sigma)=f([y^\sigma,u^\sigma])-[y^\sigma,f(u^\sigma)]=0-[y^\sigma,g^\sigma(u^\sigma)]=0$,
and, it holds that  $ [f,ad\, y^{- \sigma}](u^\sigma) =f([y^{-\sigma},u^\sigma])-[y^{-\sigma},f(u^\sigma)]= [y^{-\sigma},g^\sigma(u^\sigma)]- [y^{-\sigma},g^\sigma(u^\sigma)]=0$.
Moreover,
\begin{align*}[f,ad\, y_0](u^\sigma)&=f([y_0,u^\sigma])-[y_0,f(u^\sigma)]=\\
&=f\big(\sum\{a_i^\sigma,b_i^{- \sigma},u^\sigma\}\big)-\sum\{a_i^\sigma,b_i^{- \sigma},g^\sigma(u^\sigma)\}=\\
&=\sum    g^\sigma\big(\{a_i^\sigma,b_i^{- \sigma},u^\sigma\}\big) -\sum\{a_i^\sigma,b_i^{- \sigma}, g^\sigma(u^\sigma )\}
=0,\end{align*}  and, given $v^\sigma\in V^\sigma$,
\begin{align*}[f,ad\, y^\sigma]([u^\sigma, v^{-\sigma}])&=f( [y^\sigma,[u^\sigma,v^{-\sigma}]])- [y^\sigma,
f([u^\sigma,v^{-\sigma}])]=\\
&=-g^\sigma( \{y^\sigma,v^{-\sigma},u^\sigma\})+ \{y^\sigma,v^{-\sigma},g^\sigma(u^\sigma)\}=0,\end{align*}
and, similarly    $[f,ad\, y^\sigma]([v^\sigma, u^{-\sigma}])=0$.
Finally we have:
 \begin{align*}[f,ad\, y_0]&([u^\sigma, v^{-\sigma}])=f( [y_0,[u^\sigma,v^{-\sigma}]])- [y_0,
f([u^\sigma,v^{-\sigma}])]=\\
&=\sum f\Big( \big[ \{a_i^\sigma,b_i^{-\sigma},u^\sigma\},v^{-\sigma}\big]-
\big[u^\sigma,\{b_i^{-\sigma},a_i^\sigma,v^{-\sigma}\}\big]\Big)-\\
&\quad -[y_0,  [g^\sigma(u^\sigma),v^{-\sigma}]]=\\
&=\sum \Big( \big[g^\sigma(\{a_i^\sigma,b_i^{-\sigma} ,u^\sigma\}),v^{-\sigma}\big]-
\big[g^\sigma(u^\sigma),\{b_i^{-\sigma},a_i^\sigma ,v^{-\sigma}\}\big]\Big)-\\
&\quad -\sum \Big(\big[\{a_i^\sigma,b_i^{-\sigma} ,g^\sigma(u^\sigma)\},v^{-\sigma}\big]-
\big[g^\sigma(u^\sigma),\{b_i^{-\sigma},a_i^\sigma ,v^{-\sigma}\}\big]\Big)=0,\end{align*}
which implies that $[f,ad\, y] (  {\cal
I}(U))=0$, and therefore that  the pair
   $(f,{\cal I}(U))$ is a  permissible map
of $TKK(V)$.

Next we claim that  $\Upsilon$ is a
well-defined map.   To prove this claim  let
    $(g_1,U_1)$ and $ (g_2,U_2)$ be permissible maps of $V$,
 such that      $\mu=[g_1,U_1]= [g_2,U_2]\in {\cal
C}(V)$.  Using \cite[Lemma 1.10]{pi-ii},  let $U$ be an essential
ideal of $V$, contained into  $ U_1\cap U_2$,   such that
$(g_1)_U=(g_2)_U\in Hom_V(U,V)$. Therefore
 $[g_1,U]=[g_2,U]$ in  ${\cal C}(V)$, and  consequently
$\Upsilon([g_1,U])=\overline{(f_1,{\cal I}(U))}$ and
$\Upsilon([g_2,U])=\overline{(f_2,{\cal I}(U))}$  agree  on
${\cal I}(U)$. Hence $\Upsilon$ is well defined.

To complete the proof it suffices to prove that
   $\Upsilon$ is a two-sided inverse for  the map $\Psi$ defined in  Proposition~\ref{aplicacion-centroides}.  Let  $\mu=[g,U]\in
{\cal C}(V)$. Since    ${\cal I}(U)\cap V=U$,   by
 Lemma \ref{cambio-ideal},  it holds that
 $\Psi\Upsilon(\mu)=\Psi\big(\ \overline{\big(f,{\cal I}(U)\big)}\ \big)=[f_{V\ast U},V\ast U]$.  Thus,  since by \cite[Lemma 1.9]{pi-ii}
any permissible map  $(g,U)$ of  $V$  restricts  to $V\ast
U$,  we have   $[f_{V\ast U},V\ast
U]=[g,U]$, which implies that $\Psi\Upsilon=id_{{\cal C}(V)}$.
Conversely  given   $\lambda=\overline{(f,I)}\in{\cal C}(TKK(V))$, again using
 Lemma \ref{cambio-ideal},  we obtain    $I={\cal I}(I\cap V)$ and, therefore that
 $\Upsilon\Psi(\lambda)=\Upsilon([f_{I\cap V}, I\cap V])=\overline{(f,I)} $, which implies that
  $\Upsilon\Psi =id_{{\cal C}(TKK(V))}$.

Therefore,     $\Psi$  and $\Upsilon$ are mutually inverse  ring homomorphisms,
defining  an  isomorphism  between the extended
centroid ${\cal C}(V)$ of a nondegenerate Jordan pair $V$ and the extended
centroid ${\cal C}(TKK((V))$ of its TKK-algebra.
\end{proof}

 \section{The  central closure of the TKK-algebra of a nondegenerate Jordan pair}

We begin this section recalling some facts
  that can be found, or easily
derived, from  \cite{baxter-martindale,emo,pi-ii}.

\begin{remark}\label{representacion-anulada-dos} Let $L$ be a semiprime Lie algebra.
\begin{enumerate}\item[(i)]
Any element
  $x\in {\cal C}(L)\otimes L$  admits a (non necessarily unique)  representation
      $x=\sum\lambda_i\otimes a_i$, where
  $\lambda_i\in{\cal C}(L)$, $a_i\in L$.   A  representation   $\sum\lambda_i\otimes
 a_i$ of $x$ is    $I$-vanishing, for an essential ideal $I$ of $L$, if there exists  $(f_i,I)\in \lambda_i$ such that
  $\sum [f_i(y), p(a_i)]=0$ for all $y\in I$,
    $p\in Ad( L)$   \cite{baxter-martindale}.
\item[(ii)]
The   central closure ${\cal C}(L)L$ of $L$ is ${\cal C}(L)L=\big({\cal C}(L)\otimes L\big)/M$,
where $M$ denotes the set of all $I$-vanishing elements of  ${\cal
C}(L)\otimes L$ for some essential ideal $I$ of $L$
\cite[p.~1111]{baxter-martindale}.
   Moreover, $M$  is  the unique ideal of   ${\cal C}(L)\otimes
 L$ maximal with respect to    $R\subseteq M$ and $M\cap (1\otimes L)=0$,
 where $R$ is the ideal of ${\cal C}(L)\otimes L$ generated by all elements of the form $\lambda\otimes u-1\otimes f(u)$ with $\lambda=\overline{(f,U)}$ and $u\in U$
 \cite[Lemma 2.11]{baxter-martindale}.
  \end{enumerate}\end{remark}

 \begin{lemma}\label{representacion-anulada}  Let $L$ be the TKK-algebra of nondegenerate Jordan pair $V$.   Then
$$R= \Big\{\sum(\rho_i\lambda_i\otimes x_i-\rho_i\otimes f_i(x_i))\mid\rho_i, \lambda_i\in{\cal C}(L), (f_i,I_i)\in\lambda_i,  x_i\in  I_i \Big\} $$
 is a 3-graded ideal of   ${\cal C}(L)\otimes L$  with respect to the grading induced    in
   ${\cal C}(L)\otimes L$.  \end{lemma}

\begin{proof} The statement follows from  Remark \ref{representacion-anulada-dos} and  \cite[Lemma 3.2]{pi-ii}.\end{proof}

 \begin{theorem}\label{isomorfismo-clausuras}  Let  $V$ be a nondegenerate Jordan pair.
  Then, the
 central closure of the TKK-algebra of $V$ is isomorphic to the TKK-algebra of the extended central closure ${\cal C}(V)V$ of $V$.
 \end{theorem}

 \begin{proof} Let us denote $L=TKK(V)$  and consider the map:
\[\begin{array}{lcl}
F:\ {\cal C}(L)\times L   & \to & TKK({\cal C}(V)V)\\
\qquad (\lambda,a_\sigma) &\mapsto & \Psi(\lambda)a_\sigma\in{\cal C}(V)V^\sigma \\
\qquad(\lambda,\delta(x,y))&\mapsto & \Psi(\lambda)\delta(x,y)\in \hbox{IDer}({\cal C}(V)V) \end{array}\] %
where $\lambda\in{\cal C}(L)$, $a_\sigma\in V^\sigma$,
$\sigma=\pm$,  $\delta(x,y)\in L_0=\hbox{IDer }V$ and  $\Psi$ is the ring homomorphism defined in
  Proposition~\ref{aplicacion-centroides}.

  It is not difficult to prove that $F$ is a
  well-defined bilinear map  and, using that  $\Psi$ is a ring homomorphism,  it also follows easily that $F$ is a
   balanced map. That is,  for all
 $\lambda$,\,
 $\lambda_1$,\ $\lambda_2\in {\cal C}(L)$,\,
  $l^\sigma\in L_\sigma$,\, $l^\beta\in L_\beta$,\,  $\sigma,\beta\in\{0,\pm \}$ and $\alpha\in \Phi$,  it holds that:
 \begin{enumerate}
 \item[(i)] $ F\big((\lambda_1 +\lambda_2,l^\sigma)-(\lambda_1,l^\sigma)-( \lambda_2,l^\sigma)\big)= \Psi(\lambda_1+\lambda_2)l^\sigma
 -\Psi(\lambda_1 )l^\sigma- \Psi( \lambda_2)l^\sigma~=~0$,
 \item[(ii)] $ F\big((\lambda   ,l^\sigma_1+l^\beta_2)
 -(\lambda,l^\sigma_1 )-(\lambda ,l^\beta_2)\big) =\Psi(\lambda)(l^\sigma_1+l^\beta_2)-\Psi(\lambda)l^\sigma_1- \Psi(
 \lambda)l^\beta_2 =0$,
 \item[(iii)] $ F\big((\alpha\lambda,l^\sigma)-(\lambda,\alpha l^\sigma)\big)=\Psi(\alpha\lambda)
 l^\sigma -\Psi(\lambda)\alpha l^\sigma=0$.
 \end{enumerate}
 This results into $F$ defining
  a (3-graded) homomorphism of
3-graded Lie $\Phi$-algebras (also denoted by $F$)
 $F:{\cal C}(L)\otimes L\to TKK({\cal C}(V)V)$. Moreover, since
  ${\cal C}(V)$ is  von Neumann
regu\-lar   \cite[Theorem 1.15]{pi-ii},  ${\cal C}(V)[V^+,V^-]={\cal C}(V)^2 [V^+,V^-]=[{\cal C}(V)V^+,{\cal C}(V)V^-] $, and  $F$ is an epimorphism.

Next we claim that $Ker F$ equals the ideal $M$  described in Remark \ref{representacion-anulada-dos}(ii).     To prove  this claim  let us first consider an element   $1\otimes a\in Ker F\cap (1\otimes L)$. Then   $0=F(1\otimes a)=\Psi(1)a=a$, that results into
  $Ker F\cap (1\otimes L)=0$. Hence, by Remark~\ref{representacion-anulada-dos}(ii), to prove that   $Ker F\subseteq M$, it suffices to check that
  $R\subseteq Ker F$.

Let  $\lambda\otimes y-1\otimes f(y)\in R$, where
$\lambda=\overline{(f,I)}\in{\cal C}(L)$ and, by Remark \ref{cons-cambio-ideal}, we can assume  that $I$ is  3-graded.
 By Proposition~\ref{aplicacion-centroides}, $\Psi(\lambda)=[f_{I\cap V},I\cap
V]$, and $F (\lambda\otimes y-1\otimes f(y) ) =
 F (\lambda\otimes y)-F(1\otimes f(y) )=\Psi(\lambda)y-\Psi(1)f(y)= \Psi(\lambda)y- f(y)$.
%%%%%
Thus, for any      $ y^\sigma\in I\cap V^\sigma$,
 $\sigma=\pm$, it follows that
 $\Psi(\lambda)y^\sigma-f(y^\sigma)=f^\sigma(y^\sigma)-f(y^\sigma)
 =0$ and, therefore, for all
  $v^\sigma\in V^\sigma$,  \ $u^{-\sigma}$,\ $w^{-\sigma}\in V^{-\sigma}$,  $y^\sigma \in I^\sigma$, the element
$\Psi(\lambda)[v^\sigma,[u^{-\sigma},[w^{-\sigma},y^\sigma]]]-f\big([v^\sigma,[u^{-\sigma},[w^{-\sigma},y^\sigma]]]\big)= \Psi(\lambda)[v^\sigma,[u^{-\sigma},[w^{-\sigma},y^\sigma]]]- [v^\sigma,[u^{-\sigma},[w^{-\sigma},f(y^\sigma)]]]$
defines an inner derivation on  $ {\cal C}(V)V $, such that
 for any
    $a^\sigma\in{\cal C}(V)V^\sigma$:
\begin{align*}&\Big(\Psi(\lambda)[v^\sigma,[u^{-\sigma},[w^{-\sigma},y^\sigma]]]-
[v^\sigma,[u^{-\sigma},[w^{-\sigma},f(y^\sigma)]]]\Big)a^\sigma=\\
&=-\Psi(\lambda)\{v^\sigma,\{u^{-\sigma},y^\sigma, w^{-\sigma}\},a^\sigma\}+
    \{v^\sigma,\{u^{-\sigma},f(y^\sigma), w^{-\sigma}  \},a^\sigma\}=\\%
&=-\{v^\sigma,\Psi(\lambda)\{u^{-\sigma},y^\sigma, w^{-\sigma}\},a^\sigma\}+
   \{v^\sigma,\{u^{-\sigma},f(y^\sigma), w^{-\sigma}  \},a^\sigma\}=\\%
&=-\{v^\sigma,\{u^{-\sigma},\Psi(\lambda)y^\sigma, w^{-\sigma}\},a^\sigma\}+
     \{v^\sigma,\{u^{-\sigma},f(y^\sigma), w^{-\sigma} \},a^\sigma\}=\\%
&=-\{v^\sigma,\{u^{-\sigma},f(y^\sigma), w^{-\sigma}\},a^\sigma\}+
       \{v^\sigma,\{u^{-\sigma},f(y^\sigma), w^{-\sigma} \},a^\sigma\}=0.\end{align*}
Similarly     $
\Big(\Psi(\lambda)[v^\sigma,[u^{-\sigma},[w^{-\sigma},y^\sigma]]]-
[v^\sigma,[u^{-\sigma},[w^{-\sigma},f(y^\sigma)]]]\Big)a^{-\sigma}=0
$ holds for all  $a^{-\sigma}\in{\cal C}(V)V^{-\sigma}$, which implies, by
the maximality of  $M$ (see Remark~\ref{representacion-anulada-dos}), that
  $R\subseteq Ker F\subseteq M$.

To prove   that $M\subseteq Ker F$, take now  $x\in M$ and let  $x= \sum \lambda_i\otimes a_i$ be a
$I$-vanishing representation  of $x$, where   $I$  is assumed to be a 3-graded ideal by Remark \ref{cons-cambio-ideal}. Then
 $F(x)=\sum\Psi(\lambda_i)a_i\in TKK({\cal C}(V)V)$, where $\Psi(\lambda_i)=[(f_i)_{I\cap V},I\cap V]$,
and for all   $u^\sigma\in I\cap V^\sigma$,
    we have
\begin{align*}\big[F(x),u^\sigma\big]&=\big[\sum\Psi(\lambda_i)a_i,u^\sigma\big]=
\sum[\Psi(\lambda_i)a_i,u^\sigma ]=\\
&=\sum[a_i,\Psi(\lambda_i)u^\sigma]=\big[\sum
a_i,f_i(u^\sigma)\big]=0.\end{align*}
Moreover, for all      $v^\sigma\in V^\sigma$,  \ $v^{-\sigma}$,\
$w^{-\sigma}\in V^{-\sigma}$,   $y^\sigma \in I^\sigma$, it holds that:
\begin{align*}& \big[F(x),[v^\sigma,[u^{-\sigma},[w^{-\sigma},y^\sigma]]]\big] =\\%
&=[[F(x),v^\sigma],[u^{-\sigma},[w^{-\sigma},y^\sigma]]]+
  [v^\sigma,[[F(x),u^{-\sigma}],[w^{-\sigma},y^\sigma]]]+\\%
&+[v^\sigma,[u^{-\sigma},[[F(x),w^{-\sigma}],y^\sigma]]]+
 [v^\sigma,[u^{-\sigma},[w^{-\sigma},[F(x),y^\sigma]]]]=\\%
&=[[\sum\Psi(\lambda_i)a_i,v^\sigma],[u^{-\sigma},[w^{-\sigma},y^\sigma]]]+
 [v^\sigma,[[\sum\Psi(\lambda_i)a_i,u^{-\sigma}],[w^{-\sigma},y^\sigma]]]+\\%
&+[v^\sigma,[u^{-\sigma},[[\sum\Psi(\lambda_i)a_i,w^{-\sigma}],y^\sigma]]]+
  [v^\sigma,[u^{-\sigma},[w^{-\sigma},[\sum\Psi(\lambda_i)a_i,y^\sigma]]]]= \\%%%%%%%%%%%%%%%%
 &=\sum[[  a_i,v^\sigma],[u^{-\sigma},[w^{-\sigma},f_i(y^\sigma)]]]+
   \sum[v^\sigma,[[  a_i,u^{-\sigma}],[w^{-\sigma},f_i(y^\sigma)]]]+ \\
&+\sum[v^\sigma,[u^{-\sigma},[[  a_i,w^{-\sigma}],f_i(y^\sigma)]]]+
    \sum[v^\sigma,[u^{-\sigma},[w^{-\sigma},[  a_i,f_i(y^\sigma)]]]]=\\
 &=\sum\big[  a_i,[v^\sigma,[u^{-\sigma},[w^{-\sigma},f_i(y^\sigma)]]]\big]= \\%
&=\sum\big[  a_i,f_i\big([v^\sigma,[u^{-\sigma},[w^{-\sigma},y^\sigma]]]\big) \big]=0,\end{align*}%
since  $\sum \lambda_i\otimes a_i$  is a  $I$-vanishing
representation of     $x$. Indeed, it suffices to consider     $p(x)=x$  to be the polynomial
appearing in Remark~\ref{representacion-anulada-dos}.
This implies that   $F(x)$ belongs to  the centralizer of the ideal ${\cal I}({\cal
C}(V)(I\cap V))$  in $TKK({\cal C}(V)V)$, that vanishes, since the essentiality of $I$ in $L$, implies that of ${\cal C}(V)(I\cap V)$ in ${\cal C}(V)V$, and therefore the essentiality of  ${\cal I}({\cal
C}(V)(I\cap V))$  in  $TKK({\cal C}(V)V)$.    Hence
   $F(x)=0$ and
  $ Ker\,F=M$  follows.

Finally since  ${\cal C}(L)L= \big({\cal C}(L)\otimes
  L\big)/M$, we  obtain that
 $F:{\cal C}(L)L\ \to TKK({\cal C}(V)V)$  is an isomorphism of 3-graded Lie algebras.  \end{proof}

\section{Jordan 3-graded Lie algebras with polynomial identities. }

In this section we  begin the  study of   Jordan 3-graded Lie algebras
satisfying essential 3-graded polynomial identities.

\begin{p}\label{free lie}
Following \cite[2.7]{neher-libres},  we
denote by ${\cal L}(X)={\cal L}(X^+\cup X^-)$ the free   Lie algebra on a
polarized set $X=X^+\cup X^-$, which  admits a   $\mathbb{Z}$-grading
${\cal L}(X )=\oplus_{n\in \mathbb{Z}} {\cal L}^{(n)}(X )$,  defined by the map $\vartheta:X\to \mathbb{Z}$ given by $\vartheta(x^\sigma)=\sigma$, for all $x^\sigma\in X^\sigma$, $\sigma=\pm$. The
quotient algebra of ${\cal L}(X)$  by the ideal
  generated by all monomials:
 $$ [x_1^{\sigma_1}[x_2^{\sigma_2}[x_3^{\sigma_3}[\ldots
x_{2n}^{\sigma_{2n}}]\ldots]]],\quad\sigma_1=\sigma_2,\  \sigma_{2i-1}+\sigma_{2i}=0\ \hbox{ with }\  i\geq2,$$ where $\sigma=\pm $,
$x_i^{\sigma_i}\in X^{\sigma_i}$\ and   $n\geq1$, is a 3-graded Lie algebra
${\cal L}(X^+,X^-)={\cal L}(X^+,X^-)_1\oplus {\cal L}(X^+,X^-)_0\oplus {\cal L}(X^+,X^-)_{-1}$, where  ${\cal L}(X^+,X^-)_n$  denotes  the canonical projection of ${\cal L}^{(n)}(X )$ in ${\cal L}(X^+,X^-)$, for $n=0,\pm1$. Then ${\cal L}(X^+,X^-)$ is the  free 3-graded Lie algebra, that is, for every 3-graded Lie algebra
  $G=G_1\oplus G_0\oplus G_{-1}$ and every map   $f:X\to G$  such that  $f(X^\sigma)\subseteq G_\sigma$,\ $\sigma=\pm $,\
 there exists a unique Lie algebra homomorphism
 $F:{\cal L}(X^+,X^-)\to G$
such that  $F\circ  \iota=f$, where $\iota:X\to  {\cal L}(X^+,X^-)$ denotes the canonical map.  Moreover $F$ is  3-graded. It also holds
  (see \cite[2.7]{neher-libres}):
\begin{enumerate}
\item[(i)] The map $ \iota:X\to {\cal L}(X^+,X^-)$ is injective,
\item[(ii)] $\big({\cal L}(X^+,X^-)_1,{\cal L}(X^+,X^-)_{-1}\big) $ is a Jordan pair,
\item[(iii)] $ {\cal L}(X^+,X^-)_0=\big[{\cal L}(X^+,X^-)_1,{\cal L}(X^+,X^-)_{-1}\big]$.
\end{enumerate}
\end{p}

\begin{p}\label{def free Lie 3-graded}
A 3-graded polynomial $f(x_1^+,\ldots,x_n^+,x_1^-,\ldots,x_m^-)\in
{\cal L}(X^+,X^-)$ is a   {\it 3-graded  polynomial identity} of a
3-graded Lie algebra   $L=L_1\oplus L_0\oplus L_{-1}$  if it is
mapped to zero under every homomorphism  $ \varphi: X\to L$  such that
  $\varphi(X^\sigma)\subseteq L_\sigma$ \cite[p. 377]{z-grad}.
\end{p}

\begin{lemma}\label{lemma free objects} Let   $FJP(X^+,X^-)$  be the free Jordan pair on the polarized set $X=X^+\cup X^-$. Then:
\begin{enumerate}
\item[(i)]
  $
  FJP(X^+,X^-)\cong  \big({\cal L}(X^+,X^-)_1,{\cal L}(X^+,X^-)_{-1}\big) $.
  \item[(ii)]  ${\cal L} (X^+,X^-)$  is a central extension of $TKK\big(FJP(X^+,X^-)\big)$, that is,
      $
{\cal L}(X^+,X^-)/C_V \cong TKK\big(FJP(X^+,X^-)\big) $, where
    $$C_V=\big\{x\in {\cal L}(X^+,X^-)_0\mid [x,{\cal
L}(X^+,X^-)_1]=0=[x,{\cal L}(X^+,X^-)_{-1}]\big\}.$$ \end{enumerate}\end{lemma}

\begin{proof}  (i) follows from the universal properties of $FJP(X^+,X^-)$ (see \cite{neher-pi split pairs}) and ${\cal L}(X^+,X^-)$, and (ii) follows from \ref{central extension}  and \ref{def free Lie 3-graded}, since $
C_V= Z({\cal L}(X^+,X^-)\cap {\cal
L}(X^+,X^-)_0$. See also \cite[Lemma 2.8]{neher-libres}
  \end{proof}

 \begin{p}  Let  $Ass[X ]=Ass[X^+\cup X^-]$ be the free associative algebra on the polarized set  of generators $X=X^+\cup X^-$,
 with $\mathbb{Z}$-grading defined by $\vartheta(x^\sigma)=\sigma1$ for  all $x^\sigma\in X^\sigma$. The quotient algebra of
 $Ass[X ]= \oplus_{n\in \mathbb{Z}} Ass^{(n)}[X ]$  by the ideal generated by $\sum_{|n|>1} Ass^{(n)}[X ]$ (equivalently by
 the set $\{x^\sigma y^\sigma\mid x^\sigma, y^\sigma\in X^\sigma,\sigma=\pm\}$) is the
 free 3-graded associative algebra, denoted
 $Ass[X^+, X^-]$,   on  $X=X^+\cup X^-$ \cite[p.~352]{z-grad}.
   We denote by $S{\cal L}(X^+,X^-)$ the subalgebra of $Ass[X^+, X^-]^{(-)}$
generated by the elements of $X$. Then, $S{\cal L}(X^+,X^-)$  is the free special 3-graded Lie
algebra, and, by the universal property of ${\cal L}(X^+,X^-)$,
 there exists a unique homomorphism of 3-graded Lie
algebras   $\pi: {\cal L}(X^+,X^-) \to S {\cal L}(X^+,X^-)$
extending the inclusion  $X \subseteq S {\cal
L}(X^+,X^-)$.  \end{p}

\begin{p}
We will say that a 3-graded  polynomial  $f\in {\cal L}(X^+,X^-)$ is
 essential  if its image      $\pi(f)\in S {\cal
L}(X^+,X^-)\subseteq Ass[X^+,X^-]$ is nonzero and has a monic
leading  term (of highest    $deg$ degree) as an associative
polynomial \cite[p.~377]{z-grad}.
\end{p}

As usual, polynomial identities will be assumed to be multilinear polynomial identities.

\begin{lemma}\label{lem multilinear} Let  $L=L_1+L_0+L_{-1}$   be a 3-graded Lie algebra satisfying
an essential 3-graded polynomial identity   $f\in {\cal L}(X^+,X^-)$ of degree $d$. Then $L$ satisfies a multilinear essential
3-graded polynomial identity of degree less than or equal to $d$.\end{lemma}

\begin{proof}  See \cite[6.2.4]{herstein-carus}.\end{proof}

Next we   characterize  essential 3-graded polynomial identities in terms of the     special Lie algebras  $sl(n)$, endowed with  3-gradings defined by decompositions of the form $n=p+q$, for some $p,q\in \mathbb{N}$.

\begin{proposition}\label{prop_sln} Let $f\in {\cal L}(X^+,X^-)$ be a 3-graded polynomial.
Then,   $\pi(f)\in S{\cal L}(X^+,X^-)$ is nonzero if and only if there exist  $p$,\ $q\in \mathbb{N}$ such  that $f$ is not an
identity of  $sl(p+q)$. Thus,  if $\Phi$ is a field,    $f$ is an essential 3-graded  polynomial if and only if  there exist $p$,\ $q\in \mathbb{N}$ such that $f$ is not an identity of
$sl(p+q)$.\end{proposition}

\begin{proof}
 Let
 $f=f_1+f_0+f_{-1}\in {\cal L}(X^+,X^-)$ be a 3-graded polynomial such that
  $0\neq \pi(f)=g(x_1^+,\ldots,
x_1^-,\ldots)\in S{\cal L}(X^+,X^-)$. We can assume
$X^\sigma=\{x_1^\sigma,\ldots, x_d^\sigma\}$ for some $d\in \mathbb{N}$ and $\sigma=\pm$.

Let $k=deg(g)$  be   the degree of $g$ as an element
 of $S{\cal L}(X^+,X^-)\subseteq Ass[X^+,X^-]$, defined in the obvious way, and let   $N$  be the ideal of   $Ass[X^+,X^-]$ generated  by all   monomials of degree
  $deg$ strictly greater  that~$k$.
  Then, $N$
      is a 3-graded ideal  $N=N_1+N_0+N_{-1}$ of $Ass[X^+,X^-]$, such that $N_i=N\cap Ass[X^+,X^-]_i$, for $i=0,\pm1$,
  and the quotient algebra
 $A=Ass[X^+,X^-]/N$  is a 3-graded associative algebra.

 Write now $A=A_1+A_0+A_{-1}$, where $A_i=Ass[X^+,X^-]_i/N_i$ for $i\in \{0,\pm\}$.
 Then $  (A_1,A_{-1}) $ is an associative pair, and its   standard imbedding  ${\cal E}$  is a
 finite-dimensional free $\Phi$-module \cite[2.2]{mp-lma pares}.

Following the proof of \cite[Lemma 6.2.1]{herstein-carus}, consider now the regular
representation  $\rho :{\cal E}\to End_\Phi ({\cal   E})$    of    ${\cal E}$ as a right  ${\cal E}$-module.
The $\Phi$-module decomposition of
  ${\cal E}$, in matricial notation:
 \[{\cal E}= \left(\begin{array}{rl}
{\cal E}_{11} &{\cal E}_{12}\\
{\cal E}_{21}&{\cal E}_{22}\end{array}\right)=%
\left(\begin{array}{rl}
{\cal E}_{11} &0\\
{\cal E}_{21}&0\end{array}\right)+%
\left(\begin{array}{rl}
0 &{\cal E}_{12}\\
0&{\cal E}_{22}\end{array}\right)=M_1+M_2,\]
 results into a 3-grading of
$End_\Phi({\cal E}) =End_\Phi({\cal E})_1\oplus End_\Phi({\cal E})_0\oplus End_\Phi({\cal E})_{-1} $,
given by:
\begin{align*} &End_\Phi({\cal E})_1= Hom_\Phi(M_1,M_2),\\
&End_\Phi({\cal E})_0=End_\Phi(M_1)\oplus End_\Phi(M_2),\\
&End_\Phi({\cal E})_{-1}= Hom_\Phi(M_2,M_1), \end{align*}
that makes
   $\rho $  a 3-graded homomorphism. Moreover,  since
${\cal E}$ is an  unital algebra,  $\rho$
 is indeed a 3-graded monomorphism.

Fixing now bases of the free  $\Phi$-modules  $M_1$ and $M_{-1}$, we can obtain a
3-graded isomorphism
  $End_\Phi({\cal E})\cong M_n(\Phi)$, where $n=\dim_\Phi({\cal E})$
  and
  $M_n(\Phi)=M_n(\Phi)_1\oplus M_n(\Phi)_0\oplus M_n(\Phi)_{-1}$  is the
  3-grading given by
\newline
\[M_n(\Phi)= \left(\begin{array}{cc}
0 &M_{p,q}(\Phi)\\
0&0\end{array}\right)\oplus%
\left(\begin{array}{cc}
M_p(\Phi) &0\\
0&M_q(\Phi)\end{array}\right)\oplus%
 \left(\begin{array}{cc}
0 &0\\
M_{q,p}(\Phi)&0\end{array}\right), \] \newline
 where $p=\dim_\Phi(M_1)$ and  $q=\dim_\Phi(M_2)$. (Note that here $p=q$.)
Consequently the  homomorphism
 $\varphi: A\to M_n(\Phi)$,  resulting from  the composition of the inclusion   $A\subseteq {\cal E}$,   the regular
 representation
  $\rho: {\cal  E}\to End_\Phi({\cal E})$ and the above isomorphism  $End_\Phi({\cal E})\cong M_n(\Phi)$,
    is a 3-graded monomorphism.
      Moreover, if  we denote by    $\overline{x_i^\sigma}=x_i^\sigma +N\in
A_\sigma$   the homomorphic images of the elements $x_i^\sigma\in X^\sigma $ in the quotient algebra $A=Ass[X^+,X^-]/N$, then,
 $\varphi(\overline{x_i^\sigma})=a_i^\sigma \in M_n(\Phi)$ is traceless matrix, and, therefore
  $a_i^\sigma\in sl(p+q)$, for all $i=1,\ldots,d$ and $\sigma=\pm$.

Assume now that
 $f$ is a (multilinear, see Lemma \ref{lem multilinear}) polynomial identity   of $sl(p+q)$. Denoting by $\varphi^{(-)}:A^{(-)}\to
M_n(\Phi)^{(-)}$    the Lie algebra  monomorphism   induced by  $\varphi$, it follows that:
\begin{align*} 0&=f(a_1^+,\ldots,a_d^+,a_1^-,\ldots,a_d^-) =f(\varphi(\overline{x_1^+}),\ldots,\varphi(\overline{x_d^+}),\varphi(\overline{x_1^-}),\ldots,\varphi(\overline{x_d^-})) \\&=\varphi(f(\overline{x_1^+},\ldots,\overline{x_d^+},\overline{x_1^-},\ldots,\overline{x_d^-}))   =\varphi(\overline{f( x_1^+,\ldots,x_d^+,x_1^-,\ldots,x_d^-)})=  \varphi(\overline{g}).\end{align*}
But, since   $\varphi^{(-)}$ is a monomorphism, this implies    $\overline{g}=0$,
contradicting  the choice of   $N$.
Hence there exist $p,q\in \mathbb{N}$ such that $f$ is not an identity of $sl(p+q)$.

Conversely, let us suppose  that there exist   $p$,\ $q\in \mathbb{N}$ such that   $f\in {\cal
L}(X^+,X^-)$  is not a polynomial  identity for $sl(p+q)$. Writing
   $n=p+q$, this induces a  3-grading in
    $sl(n)$   such that
 $\big(sl(n)_1,sl(n)_{-1}\big) \subseteq \big(M_{p,q}(\Phi),M_{q,p}(\Phi)\big)
 $
   is a special Jordan pair.

Consider now  the case $f=f_+\in  {\cal L}(X^+,X^-)_1$.  Then, since  $f=f_+ $ is not a polynomial identity of $sl(n)$, by Lemma \ref{lemma free objects}, it follows that  $f=f_+\in
FJP[X^+,X^-]^+$ is not
 a polynomial identity for its associated Jordan pair $\big(sl(n)_1,sl(n)_{-1}\big)$, and therefore $f=f_+$
  has  a nonzero image in     $FSJP[X^+,X^-]$,  the free special Jordan pair.
Hence
  $\pi(f)=\pi(f_+)\neq0$.  The case when $f=f_-\in  {\cal L}(X^+,X^-)_{-1}$  follows analogously.

Suppose next that  $f=f_0\in {\cal L}(X^+,X^-)_0$. We claim that there exists
exists    $y^\sigma\in X^+\cup X^-$ such that
$[f_0,y^\sigma]\neq0$. Indeed, relabeling if necessary,  if  $x_1^{\sigma_1}\ldots
x_{2d}^{\sigma_{2d}}$  is    a monic monomial of highest degree    in
$f_0$ and $y^\sigma$
   a variable  not  appearing in $f_0$ with    $\sigma\neq \sigma_1$ (or $\sigma\neq
   \sigma_{2d}$),  it  suffices to note that the polynomial
$[f_0,y^\sigma]$ contains the   monomial  $y^\sigma
x_1^{\sigma_1}\ldots x_{2d}^{\sigma_{2d}}$ with coefficient 1.

We claim that there exists  a large enough nonnegative integer $m_0$ such that neither $f$ nor  $[f_0,y^\sigma]$ vanish in $sl(m_0)$.  Indeed, the statement for $f$ is clear, since   it does not vanish in any  $sl(m)$  $m\geq n=p+q$, and for  $[f_0,y^\sigma]$ it follows from the case $f=f_\sigma$.

Finally,   replacing $f_0$ by  $g_\sigma=[f_0,y^\sigma]$, if necessary,  we
obtain an  essential polynomial   such that
$\pi(f)=\pi(f_0)\neq~0$.
\end{proof}

We refer to \cite[0.12]{pi-ii}, and references therein,  for the notion of essential polynomial in Jordan systems.

\begin{theorem}\label{th identidad-esencial} Let  $L$ be a Jordan 3-graded Lie algebra
with associated Jordan pair $V$. If $L$ satisfies an essential 3-graded polynomial identity, then the Jordan pair $V$ satisfies an
essential polynomial identity.
\end{theorem}

 \begin{proof}  By  Lemma \ref{lemma free objects} and Proposition \ref{prop_sln},   it suffices to note that if
  $L$ satisfies an essential
 3-graded polynomial identity
    $f=f_1+ f_0+ f_{-1} $, then $V$ satisfies the essential polynomial identity
  $g=(g_+,g_-)=\big(f_1+[f_0,y^+],f_{-1}+[f_0,y^-]\big)$.
\end{proof}

\section{Posner Rowen's theorem for Jordan 3-graded Lie algebras}

In this  section we attempt to describe strongly prime  Jordan 3-graded Lie algebras
satisfying essential 3-graded polynomial identities. To do this we  relate their polynomial identities to those of their associated Jordan pairs, to then taking advantage of the results on PI Jordan pairs   proved  in \cite{id-pares}.

In this section we will make use of
  Zelmanov's classification of
simple (infinite-dimensional) Lie algebras with a finite
nontrivial $ \mathbb{Z}$-grading  \cite{z-grad}. Therefore from now  $\Phi$ is assumed to be  a field of characteristic zero or
characteristic at least 5.

\begin{remark}\label{grado-identidad} Let  $L$ be a Jordan 3-graded Lie algebra.
If $L$ satisfies an essential
3-graded polynomial identity
   $f=f_1+ f_0+ f_{-1} $, then, by  Theorem \ref{th identidad-esencial},   its associated Jordan pair $V$ satisfies the essential polynomial identity
 $g=(g_+,g_-)=\big(f_1+[f_0,y^+],f_{-1}+[f_0,y^-]\big)$.  Since  $deg(f)={\rm max }\ \{ deg(f_1)\, ,deg(f_0)\, ,deg(f_{-1})\}$, it holds that:
  \begin{enumerate}
\item[(i)] If  $deg(f) =2d-1$  is odd,  then, for some $\sigma=\pm$,
 $deg(f)=deg(f_\sigma)>deg(f_0)$. Thus
 $deg(f)\geq deg(f_0)+1$, and we obtain that
 $deg(g)=\max \{ deg(g_+)\, ,deg(g_-)\}=deg(f )=2d-1$.
\item[(ii)] If  $deg(f) =2(d-1)$ is even, then  $deg(f)=deg(f_0)>deg(f_\sigma)$, for $\sigma=\pm$.  This results into
 $deg(f )\geq deg(f_\sigma ) +1$
and therefore
 $deg(g)=deg(f)+1=2d-1$.\end{enumerate}
Hence, if  $L$  satisfies an essential 3-graded
polynomial identity $f$  of degree  either    even  $2(d-1)$  or odd $2d-1$,
  we can assume that  its associated Jordan pair  $V$ satisfies an essential
polynomial identity  $g=(g_+,g_-)=\big(f_1+[f_0,y^+],f_{-1}+[f_0,y^-]\big)$ of degree $2d-1$.\end{remark}

 \begin{theorem}\label{teorema-posner-rowen} {\rm(Posner-Rowen's Theorem for Jordan 3-graded Lie algebras.)} Let $L$ be a Jordan 3-graded Lie algebra with
  associated Jordan pair  $V$  over a field $\Phi$ of characteristic zero or prime    $p\geq    5$. If  $L$ is
strongly prime and satisfies an essential 3-graded polynomial
identity,
  then its central closure    is simple and, therefore,     isomorphic  to one of following Lie algebras:
 \begin{enumerate}\item[{\rm I.}] $[R^{(-)},R^{(-)}]/Z$, where $R=R_{-1}+R_0+R_1$  is a  simple associative 3-graded algebra,
 finite dimensional over its center,  and
$Z$ is the center of the derived algebra  $[R^{(-)},R^{(-)}]$.
\item[{\rm II.}] $[K(R,\ast),K(R,\ast)]/Z$, where
$R=R_{-1}+R_0+R_1$  is a simple associative  3-graded  algebra,
finite dimensional over its center,
 with involution
 $\ast: R\to R$,  such that   $R_i^\ast=R_i$ for all $i\in\{0,\pm1\}$,  and
$K(R,\ast)=\{a\in R\mid a^\ast=-a\}$. \item[{\rm III.}] The
Tits-Kantor-Koecher algebra of the Jordan algebra of a
symmetric bilinear form. \item[{\rm IV.}] An exceptional Lie
algebra of   type  $E_6$ or $E_7$.\end{enumerate}
Moreover, in cases  I and II  the isomorphism preserves the
grading, that is, it is an isomorphism of 3-graded algebras.
\end{theorem}

 \begin{proof}  We first note that being    strongly prime,  then  $L$
 is isomorphic to the TKK-algebra of is associated Jordan pair   $V$, and $V$     is strongly prime (see \ref{central extension} and \cite[Proposition 11.25]{FL-libro}).

By Theorem
   \ref{th identidad-esencial}, $V$ satisfies an essential multilinear
polynomial identity  and, then, by \cite[Proposition~4.6]{id-pares}, $V$ is
homotope-PI.
Therefore, by the Jordan pair analogous of Posner Rowen  theorem
   \cite[Theorem 6.1(ii)]{pi-ii},
  the  extended central closure
 ${\cal C}(V)V$  of $V$ is a simple Jordan pair with  finite capacity.  Moreover, assuming  the degree of the essential polynomial identity of $L$ to be as in Remark \ref{grado-identidad},  and considering that $V$ and ${\cal C}(V)V$ satisfy the same multilinear polynomial identities, we can assume that ${\cal C}(V)V$ has capacity at most $2d$.

 Now,  since by Theorem
   \ref{isomorfismo-clausuras}, the central closure ${\cal C}(L)L$
of $L$ is isomorphic to the TKK-algebra
  $TKK({\cal C}(V)V)$  of the  extended central closure
 ${\cal C}(V)V$  of $V$,  by  \cite[Proposition 11.25]{FL-libro},
 ${\cal C}(L)L$   is a simple Jordan 3-graded Lie
  algebra, and therefore  ${\cal C}(L)L$  is    isomorphic to one of the  Lie algebras listed in \cite[Theorem 1]{z-grad}.

To finish the proof it suffices to note that the associative algebras appearing in cases I and II
are finite dimensional over their centers
by \cite[p. 57]{jac-pi},  and that, by
\cite[7.2, 7.3]{neher-118},  in the exceptional case IV the
only possibilities allowed for   ${\cal C}(L)L$   are types  $E_6$ or $E_7$.
\end{proof}

\section{PI Jordan 3-graded Lie algebras}

We  devote this last section to   arbitrary Jordan 3-graded Lie algebras satisfying essential 3-graded polynomial identities. To cope with the absence of regularity conditions we will consider the  Kostrikin
radical of the Lie algebras.

\begin{p}\label{radicals} An   element $z$ of a Lie algebra $L$ is a {\it crust of a thin sandwich} if $(ad\,z)^2=0$.  Lie algebras
 with no nonzero crusts of thin sandwiches are  nondegenerate Lie algebras (also called  strongly nondegenerate  in the sense of
Kostrikin). The smallest ideal of a Lie algebra $L$ that provides a nondegenerate quotient algebra is the  Kostrikin
radical  of $L$,    denoted by $K(L)$ \cite{k}.
We also recall here that the Jordan counterpart, that is, the  smallest ideal of a Jordan pair $V$ that provides a nongedenerate Jordan system is the McCrimmon radical $Mc(V)$ of $V$  \cite[p.~538-539]{z-radical}.
\end{p}

\begin{remark}  Given an ideal   $I$ of a Lie algebra $L$ we will denote
$\widetilde{I}=\{x\in L\mid [x,L]\subseteq I\}$   the anti-image of the center
    $Z(L/I)$ of the quotient Lie algebra $L/I$ by the canonical projection $L\to L/I$.
\end{remark}

\begin{proposition}\label{pr-lema-uno}  Let $L$ be a   Jordan 3-graded Lie algebra   with associated
Jordan pair  $V$. Then $K(L)=\widetilde{{\cal
I}(Mc(V))}$.\end{proposition}

\begin{proof}  Let   ${\cal I}(Mc(V))$,  defined as in \ref{cambio-ideal}, be the ideal of $L$ generated by the McCrimmon radical $ Mc(V)$  of the Jordan pair $V$.
Then
 $\overline{L}=L/{\cal I}(Mc(V))$ is a Jordan 3-graded Lie algebra, whose associated Jordan pair
  $(\overline{L}_1,\overline{L}_{-1})\cong V/Mc(V)$ is nondegenerate.

Next we claim that  $L/\widetilde{{\cal I}(M(V))}=\overline{L}/C_{(\overline{L}_1,\overline{L}_{-1})}=TKK(\overline{L}_1,\overline{L}_{-1}) $,
 where by   \ref{central extension}, it suffices to check that
 $$ \widetilde{{\cal I}(Mc(V))}/{\cal I}(M(V))=Z(L/{\cal I}(Mc(V)))=C_{(\overline{L}_1,\overline{L}_{-1})}.$$
To prove this claim let
   $\overline{z}_i\in Z(L/{\cal I}(Mc(V)))\cap
\overline{L}_i$,   $i\neq0$. Then  $\overline{z}_i$ is
 an absolute zero divisor in  $(\overline{L}_1,\overline{L}_{-1})$, but since
 $Mc(\overline{L}_1,\overline{L}_{-1})=0$, it follows that
 $\overline{z}_i=0$.  Hence  $Z(L/{\cal I}(Mc(V)))\subseteq
\overline{L}_0$ which implies that
$$ \widetilde{{\cal I}(Mc(V))}/{\cal I}(Mc(V))=Z(L/{\cal I}(Mc(V)))=C_{(\overline{L}_1,\overline{L}_{-1})}.$$
Hence  $ L/\widetilde{{\cal I}(Mc(V))}=
TKK(\overline{L}_1,\overline{L}_{-1}) $, which is a nondegenerate Lie algebra  by  \cite[Proposition 11.25]{FL-libro}, and, therefore, by the minimality of
  the Kostrikin radical,  $K(L)\subseteq \widetilde{{\cal I}(Mc(V))}$ holds.

Conversely,  it is not difficult to prove that the   Kostrikin   radical  $K(L)$   is
a 3-graded ideal  of $L$, and therefore  $L/K(L)$ is a nondegenerate Jordan
3-graded Lie algebra,  with nondegenerate
associated Jordan   pair $ W=\big(
L_1/K_1(L),L_{-1}/K_{-1}(L)\big) $. Moreover   $C_W=Z(L/K(L))=0$,  which implies, by    \ref{central extension},
  that $L/K(L)=TKK(W)$.  Thus  $\big(
L_1/K_1(L),L_{-1}/K_{-1}(L)\big)$ is a nondegenerate Jordan pair
and
 $\big( Mc(L_1),Mc(L_{-1})\big) \subseteq \big(  K_1(L),
 K_{-1}(L)\big)$ holds.
Hence   it follows that $ {\cal I}(Mc(V))\subseteq K(L)$ and,  therefore,   $ \widetilde{{\cal
I}(Mc(V))}\subseteq K(L)$  \cite{z-radical,z-grad}.
\end{proof}

\begin{lemma}\label{pr-lema-dos}  Let $L$ be a   Jordan 3-graded Lie algebra   with associated
Jordan pair  $V$.  Then:
$$\widetilde{{\cal I}(Mc(V))}=%
\bigcap \big\{\widetilde{ {\cal I}(P)}\mid   \hbox{$P$ is a strongly prime ideal of $V$}\ \big\}.$$
\end{lemma}

\begin{proof} Since, by  \cite{mc-radical},
 $Mc(V)= \bigcap \big\{P \mid   \hbox{$P$ is a strongly prime ideal of $V$}\ \big\}$,   it is straightforward that
 $\widetilde{ {\cal I}(Mc(V))  }\subseteq \widetilde{{\cal I}(P)} $  for all strongly prime ideals $P$ of $V$.
    To prove the reverse containment,  let
 $$x_i\in \bigcap \big\{\widetilde{ {\cal I}(P)}\mid   \hbox{$P$ is a strongly prime ideal of $V$}\ \big\}\cap
 L_i,\quad i\in\{\pm1,0\}.$$
If  $i=\pm1$,  by \ref{ideales-relacion-p},   $x_i\in \widetilde{{\cal I}(Mc(V))}_i$.   Otherwise  $i=0$ and then, for any strongly prime ideal $P$ of $V$ it holds that
 $[x_0,L]\subseteq {\cal I}(P)$, which implies
$[x_0,L_1+L_{-1}]\subseteq Mc(V)$, and   therefore  that  $x_0\in C_{V/Mc(V)}$.
Hence  $x_0\in \widetilde{{\cal I}(Mc(V))}$.\end{proof}

\begin{proposition}\label{producto-subdirecto-lie-jordan} Let $L$ be Jordan 3-graded Lie algebra.  Then
$L/K(L)$ is a subdirect product of strongly prime Jordan 3-graded Lie algebras.\end{proposition}

\begin{proof}  By Proposition \ref{pr-lema-uno} and  Lemma \ref{pr-lema-dos},  it  now suffices  to note that for  any strongly prime ideal
  $P$ of the Jordan pair  $V=(L_1,L_{-1})$, the quotient algebra  $L/ \widetilde{ {\cal I}(P)}$
    is a strongly prime Jordan 3-graded Lie algebra. \end{proof}

 \begin{theorem} Let  $L$ be a Jordan 3-graded Lie algebra. If $L$ satisfies an
 essential 3-graded polynomial identity,  then  the nondegenerate Jordan 3-graded Lie algebra $L/K(L)$
 is a
 subdirect product
of strongly prime Jordan 3-graded Lie algebras  satisfying the
same essential 3-graded polynomial identity. Therefore, the
central closure of each subdirect factor is isomorphic to one of the algebras   listed in Theorem \ref{teorema-posner-rowen}.
\end{theorem}

 \begin{proof}  Let $K(L)$ be the Kostrikin radical of the Jordan 3-graded Lie algebra $L$. Then, by \ref{radicals}, the quotient Lie  algebra
   $L/K(L)$ is   nondegenerate, and  it is Jordan 3-graded by Proposition \ref{pr-lema-uno}, since as noted before
      $K(L)$ a 3-graded ideal of $L$.
     Moreover, by Proposition \ref{producto-subdirecto-lie-jordan}, $L/K(L)$
  is a subdirect product of strongly prime Jordan 3-graded Lie algebras     $\{L_\lambda\}_{\lambda\in \Lambda}$, all them
 satisfying the same essential 3-graded polynomial identities  as the Lie algebra   $L$. Hence, for each   $\lambda$, the
  central closure ${\cal C}(L_\lambda)L_\lambda$ of  $L_\lambda$ is isomorphic to one of the algebras listed in Theorem
     \ref{teorema-posner-rowen}.\end{proof}

 \section*{Declaration of competing interest}

None.

%%%%%%%%%%%%%%%%%%%%%%%%%%%%%%%%%%%%%%

\end{document}